\documentclass[11pt,a4paper]{article}
\usepackage{tikz}
\usepackage{subfigure}
\usepackage[english]{babel}
\usepackage{graphicx}
\usepackage{extarrows}
\usepackage{color}
\usepackage{indentfirst, latexsym, bbm, bm, amssymb}
\usepackage{hyperref}
\usepackage{authblk}
\usepackage[top=2.5cm,bottom=2.5cm,left=2.2cm,right=2.2cm]{geometry}
\usepackage{graphicx}
\hypersetup
{
colorlinks=true,
linkcolor=blue,
filecolor=blue,
urlcolor=blue,
citecolor=cyan,
}

\newtheorem{theorem}{Theorem}[section]
\newtheorem{corollary}[theorem]{Corollary}

\newtheorem{conjecture}[theorem]{Conjecture}

\newtheorem{lemma}[theorem]{Lemma}

\newtheorem{claim}{Claim}
\newenvironment{proof}{\noindent {\bf Proof.}}{\hfill \rule{3mm}{3mm}\par\medskip}

\allowdisplaybreaks[4]

\begin{document}
\title{Stability results for Berge-matching in hypergraphs
\thanks{Jia-Bao Yang was supported by National Key R\&D Program of China under grant number 2024YFA1013900, NSFC under grant number 12471327. Leilei Zhang was  supported by JSPS KAKENHI Grant Number 25KF0036, the NSF of Hubei Province Grant Number 2025AFB309,  the China Postdoctoral Science Foundation  Grant Number 2025M773113, the Fundamental Research Funds for the Central Universities, Central China Normal University Grant Number CCNU24XJ026.}}

\author[1]{Jia-Bao Yang\footnote{Email: \texttt{jbyang1215@163.com}}}
\author[2]{Leilei Zhang\footnote{Corresponding author. Email: \texttt{mathdzhang@163.com}}}

\date{}

\affil[1]{\small School of Mathematics, Nanjing University, Nanjing, 210093, China }
\affil[2]{\small Faculty of Environment and Information Sciences, Yokohama National University, Yokohama 240-8501, Japan}

\maketitle

\begin{abstract}
Given a graph $F$, a hypergraph is called a \text{Berge-$F$} if it can be obtained by expanding each edge of \( F \) into a hyperedge containing it.  Let $M_{k}$ denote the matching of size $k$.  Kang, Ni, and Shan \cite{Kang2022} determined the Turán number of \text{Berge-$M_k$}.   Our main result shows that if an $r$-uniform hypergraph $\mathcal{H}$ on $n$ vertices has nearly as many edges as the extremal in their theorem without containing \text{Berge-$M_k$}, then $\mathcal{H}$ must be structurally close to certain well-specified graphs. Meanwhile, our result also implies several stability results, such as the stability version of the well-known Erd\H{o}s-Gallai theorem (Erd\H{o}s and Gallai, 1959 \cite{Erdos1959}).
\end{abstract}

{{\bf Keywords:} Tur\'an number; matching; Berge hypergraph; stability.}

\section{Introduction}
A hypergraph $\mathcal{H}=(V,E)$ consists of a set $V$ of vertices and a set $E$ of hyperedges.
An $r$-uniform hypergraph, or simply an $r$-graph, is a collection of $r$-subsets of a finite set. 
Given $r$-graphs $\mathcal{H}$ and $\mathcal{F}$, we say that $\mathcal{H}$ is $\it\mathcal{F} \text{-} free$ if it contains no copy of $\mathcal{F}$ as a subgraph.
The {\it Tur\'{a}n number} of $\mathcal{F}$, denoted by $\mbox{ex}_r(n,\mathcal{F})$, is the
maximum number of hyperedges in an $\mathcal{F}$-free $r$-graph on $n$ vertices. 
We say that $\mathcal{H}$ is $\it extremal$ if $\mathcal{H}$ is an $n$-vertex $\mathcal{F}$-free $r$-graph with exactly $\mbox{ex}_r(n,\mathcal{F})$ hyperedges.

The study of the behavior of $\mbox{ex}_{r}(n,\mathcal{F})$ and $\mbox{EX}_{r}(n,\mathcal{F})$ is a central topics in extremal graph theory. 
In 1941, Tur\'an proved that $\mbox{ex}_{2}(n,K_{p+1})=T(n,p)$, 
where $K_{p+1}$ denotes the complete graph of order $p+1$ for $p\geq 2$, 
and $T(n,p)$ denotes the balanced $p$-partite Tur\'an graph on $n$ vertices.
The Erd\H{o}s-Stone-Simonovits Theorem states that for any graph $G$ with chromatic number $\chi(G)= k + 1 \geq 2$, 
we have $\mbox{ex}_{2}(n,G) =(1 - 1/k + o(1))(n^2/2)$.
It is clear that the Erd\H{o}s-Stone-Simonovits theorem gives the asymptotic values of the Tur\'an numbers of all graphs with chromatic number at least three.
Thus, determining the exact Tur\'an numbers of non-bipartite graphs remains a challenging problem.

Denote by $M_{k}^r$ the matching of size $k$ in $r$-graph. 
For simplicity, we often write $M_{k}^2$ as $M_{k}$. 
In \cite{Erdos1959}, Erd\H{o}s and Gallai proved a famous result on the Tur\'an number of matchings as follows.

\begin{theorem}[Erd\H{o}s and Gallai \cite{Erdos1959}] \label{Erdos-r=2}
Let $n$ and $k$ be integers with $n\geq 2k+1$. Then 
$$\emph{ex}_2(n,M_{k+1})\leq \max\left\{\binom{2k+1}{2},\binom{k+1}{2}+(n-k-1)k\right\}.$$
\end{theorem}

As a significant extension of the Erd\H{o}s--Gallai theorem, Wang \cite{Wang2020} generalized this result to $s$-cliques. 
Later, Duan, Ning, Peng, Wang, and Yang \cite{Duan2020} determined the maximum number of $s$-cliques in graphs with matching number at most $k$ and with a given minimum degree. 
They also proved two stability results for their theorem in the case $s=2$. 
In particular, they showed that if a graph $G$ has a number of edges close to the maximum value given by their theorem, then $G$ must be a subgraph of a well-specified family of graphs. 
Very recently, using the Gallai--Edmonds theorem, Yang and Yuan \cite{Yang2024} completely resolved the stability problem for the number of cliques in graphs with a given matching number and minimum degree.

In the setting of hypergraphs, Erd\H{o}s \cite{Erdos1965} proposed a famous conjecture on matchings, which remains one of the central open Tur\'an problems in hypergraph theory.
\begin{conjecture}[Erd\H{o}s \cite{Erdos1965}]\label{Erdos-conjecture}
For integers $n\geq rk$, $r\geq 2$ and $k\geq 1$,
$$\emph{ex}_r(n,M_{k+1}^r)=\max\left\{\binom{r(k+1)-1}{r},\binom{n}{r}-\binom{n-k}{r}\right\}.$$
\end{conjecture}

For this topic, Theorem~\ref{Erdos-r=2} implies that the conjecture is settled for $r=2$.
After that, \L uczak and Mieczkowska \cite{Luczak2014} proved that the Erd\H{o}s Matching Conjecture holds for $r=3$ and sufficiently large $n$.
The case $r=3$ was later completely resolved by Frankl \cite{Frankl2017}.
Recently, Frankl and Kupavskii \cite{Frankl2019} obtained a stability result for the Erd\H{o}s Matching Conjecture 
when $r \ge 3$ and either $n \ge r(k+\max\{25,2k+2\})$ or $n \ge (2+o(1))rk$.
For other results on the Erd\H{o}s Matching Conjecture, we refer the reader to 
\cite{Alon2012,Bollobas1976,Frankl2022,Frankl2019,Frankl2012}.
The exact value of the Tur\'an number of hypergraph matchings remains unknown in general.

To state other results, we first recall a definition given by Gerbner and Palmer \cite{Gerbner2017}.
Let $F$ be a graph. A hypergraph $\mathcal{H}$ is called a \emph{Berge-$F$} 
if there exists a bijection $\phi: E(F) \to E(\mathcal{H})$ such that $e \subseteq \phi(e)$ for each $e \in E(F)$.
It is worth noting that Berge-$F$ is actually a family of hypergraphs.
We say that a hypergraph $\mathcal{H}$ is \emph{Berge-$F$-free} if it contains no subgraph isomorphic to any Berge-$F$.
By definition, $\mathrm{ex}_r(n,\text{Berge-}F)$ is the maximum number of hyperedges in an $n$-vertex Berge-$F$-free $r$-graph.

Given two $r$-graphs $\mathcal{H}_1$ and $\mathcal{H}_2$ with disjoint vertex sets, 
let $\mathcal{H}_1 \cup \mathcal{H}_2$ be the $r$-graph with vertex set $V(\mathcal{H}_1) \cup V(\mathcal{H}_2)$ and edge set $E(\mathcal{H}_1) \cup E(\mathcal{H}_2)$. 
For a positive integer $k$, we use $k\mathcal{H}$ to denote $\bigcup_{i=1}^k \mathcal{H}$.
Denote by $\mathcal{H}_1+\mathcal{H}_2$ the graph obtained from $\mathcal{H}_1\cup \mathcal{H}_2$ 
by adding every $r$-subset of $V(\mathcal{H}_1)\cup V(\mathcal{H}_2)$ 
that has nonempty intersection with both $V(\mathcal{H}_1)$ and $V(\mathcal{H}_2)$.

For $\vec{c}=(c_1,\ldots,c_m)$ with odd numbers satisfying 
$c_1 \geq c_2 \geq \cdots \geq c_m \geq 1$, 
let $\alpha(\vec{c})=\sum_{i=1}^{m} c_i$ and $\beta(\vec{c})=\sum_{i=1}^{m} \lfloor c_i/2 \rfloor$. 
Define $K(\vec{c})$ to be the graph consisting of vertex-disjoint cliques of orders 
$c_1,\ldots,c_m$. 
Denote by $H(n,k,t,\vec{c})$ the $n$-vertex graph $K_t + K(\vec{c})$, 
where 
$t+\alpha(\vec{c}) = n$ and $t+\beta(\vec{c}) = k$ (see Figure~\ref{fig:1}). 
For convenience, let 
$H(n,k,t) = H\bigl(n,k,t,(2k-2t+1,1,\ldots,1)\bigr)$. 
Let $N(K_r,G)$ denote the number of $r$-cliques in a graph $G$. 
Define 
\[
h_r(n,k,t)=N(K_r,H(n,k,t))
= {2k+1-t \choose r} + (n-2k-1+t){t \choose r-1},
\]
and 
\[
f_r(n,k,t)
= {2k+1-t \choose r} + (n-2k-1+t)t .
\]

\begin{figure}
\begin{center}
    \begin{tikzpicture}[scale = 0.5]
    \filldraw[fill=gray] (2,0) ellipse (2.25 and 0.5);
    \filldraw[fill=gray] (7,0) ellipse (1.75 and 0.5);
    \filldraw[fill=gray] (6,4) ellipse (2.5 and 0.5);
     \foreach \xa / \xb  in { 1 / 4, 2 / 4, 3 / 4, 4 / 4,   6 / 4, 7 / 4,  9 / 4, 10 / 4    2.5/ 7, 3.5 / 7, 4.5 / 7}
       \foreach \xa / \xb in {4.5/ 4, 6 / 4, 7.5 / 4}
        \foreach \ya / \yb in {0 / 0, 1 / 0, 2 / 0, 3 / 0, 4 / 0,   6 / 0, 7 / 0, 8 /0,   10 / 0, 12 / 0 }
        {\draw (\xa,\xb) -- (\ya,\yb) ;}
    \draw node at (-3,0){$K(\vec{c})$};
    \draw node at (-3,4){$K_t$};
    \foreach \ya / \yb in {0/ 0, 1 / 0, 2 / 0, 3 / 0, 4 / 0,   6 / 0, 7 / 0, 8 /0,   10 / 0, 12 / 0 }
    {\draw (\ya,\yb) circle (2pt);}
     \foreach \xa / \xb in {4.5/ 4, 6 / 4, 7.5 / 4}
       {\draw (\xa,\xb) circle (2pt);}
    \end{tikzpicture}
    \caption{ $H\big(13,6,3,(5,3,1,1)\big)$}
    \label{fig:1}
 \end{center}
\end{figure}
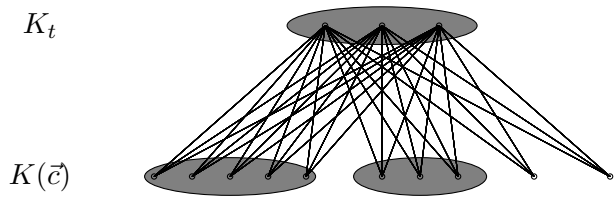

Given a graph $G$, the $r$-clique hypergraph of $G$, denoted by $\mathcal{G}^r$, 
is the $r$-graph whose vertex set is $V(G)$ and in which an $r$-subset of $V(G)$ 
forms a hyperedge of $\mathcal{G}^r$ if and only if the $r$-subset is an $r$-clique of $G$.

In 2022, Kang, Ni, and Shan \cite{Kang2022} proved the following theorem, 
which extends Theorem~\ref{Erdos-r=2} and determines the Tur\'an number of 
Berge-matchings in hypergraphs.

\begin{theorem}[Kang, Ni, and Shan \cite{Kang2022}]\label{berge-matching}
If $r\leq k-1$,
then 
$$\emph{ex}_r(n,\emph{Berge-}M_{k+1})=
\max\left\{ h_r(n,k,0),h_r\left(n,k,k\right) \right\}.$$
Moreover, the extremal hypergraphs are $\mathcal{H}^r(n,k,0)$ or $\mathcal{H}^r(n,k,k)$,
where $\mathcal{H}^r(n,k,s)$ is the $r$-clique hypergraph of $H(n,k,s)$.
\end{theorem}

Motivated by Theorem~\ref{Erdos-r=2}, Conjecture~\ref{Erdos-conjecture}, 
and the corresponding stability results, we aim to establish the stability 
versions of Theorem~\ref{berge-matching}.
It is obvious that $\mathcal{H}$ is a subgraph of $K_n^r$ if $\mathcal{H}$ is a 
Berge-$M_{k+1}$-free $r$-graph on $n \le 2k+1$ vertices, 
where $K_n^r$ is the complete $r$-graph on $n$ vertices. 
Hence, throughout this paper we assume that $n \ge 2k+2$.

For a fixed graph $G$, we color the edges of $G$ red and blue, 
and call the resulting graph a \emph{red-blue graph}.
In a red-blue graph $G$, we denote by $G_{\text{red}}$ the subgraph spanned by 
the red edges and by $G_{\text{blue}}$ the subgraph spanned by the blue edges.
Given a positive integer~$r$ and a red-blue graph $G$, define
\[
g_r(G)= e(G_{\text{red}}) + N(K_r, G_{\text{blue}}),
\]
where $N(K_r, G_{\text{blue}})$ is the number of copies of $K_r$ contained in $G_{\text{blue}}$.

The following theorem presents a stability version of Theorem~\ref{berge-matching}.

\begin{theorem}[Stability]\label{stabilitytheorem-for-bergematching}
Let $p,q,r$ be integers with $p,q \ge 0$ and $r \le k-1$.
Let $\mathcal{H}$ be a Berge-$M_{k+1}$-free $r$-graph on $n$ vertices.
If $p \le k-q$ and 
\[
|E(\mathcal{H})| > \max\{h_r(n,k,p),\, h_r(n,k,k-q),\, f_r(n,k,p),\, f_r(n,k,k-q)\}, 
\]
then there exist a hyperedge set $E \subseteq E(\mathcal{H})$ and an $M_{k+1}$-free red-blue graph $G$ 
which is a subgraph of $H(n,k,t,\vec{c})$ for some $(t,\vec{c})$ satisfying either $0 \le t \le p-1$ or $k-q+1 \le t \le k$, 
such that $|E| = e(G_{\text{red}})$ and 
$E(\mathcal{H}) \setminus E \subseteq E(\mathcal{G}^r_{\text{blue}})$.
Furthermore, if $r=2$, then we have $G = \mathcal{H}$ and $E = \emptyset$.
\end{theorem}

This stability theorem implies Theorem~\ref{berge-matching}, as discussed in Section~3.
As a direct application, we establish the following corollary, 
which is a stability version of Theorem~\ref{Erdos-r=2}.

\begin{corollary}\label{corollary1}
Let $p,q,k$ be integers with $p \le k-q$.
If $G$ is an $M_{k+1}$-free graph on $n$ vertices with 
\[
|E(G)| > \max\{h_2(n,k,p),\, h_2(n,k,k-q)\}, 
\]
then $G$ is a subgraph of $H(n,k,t,\vec{c})$ for some $(t,\vec{c})$ satisfying either $0 \le t \le p-1$ or $k-q+1 \le t \le k$.
\end{corollary}


Corollary~\ref{corollary1} also implies a stability result for the number of edges 
in a graph with a given matching number and minimum degree.
This extends the stability results established by Duan, Ning, Peng, Wang, and Yang \cite{Duan2020}.
For sufficiently large $n$, we can provide a more precise description of 
Theorem~\ref{stabilitytheorem-for-bergematching}.

\begin{theorem}[Stability]\label{stabilitytheorem-for-bergematching-sufficiently}
Let $k \ge 2q+3$ and $r \le k-q-1$.
Let $\mathcal{H}$ be a $\emph{Berge-}M_{k+1}$-free $r$-graph on $n$ vertices.
For $n$ sufficiently large, if 
\[
|E(\mathcal{H})| > h_r(n,k,k-q),
\] 
then there exists a vertex set $S \subseteq V(\mathcal{H})$ of size $k-q \le |S| \le k$ 
such that the hypergraph 
\(\mathcal{H}_0\) is $\emph{Berge-}M_{k+1-|S|}$-free, 
with vertex set $V(\mathcal{H}_0) = V(\mathcal{H}) \setminus S$ and edge set 
\[
E(\mathcal{H}_0) = \{ e \setminus S : e \in E(\mathcal{H}),\ |e \cap (V(\mathcal{H}) \setminus S)| \ge 2 \}.
\]
\end{theorem}

Note that $\mathcal{H}_0$ is not necessarily an $r$-graph, and parallel edges are allowed in $\mathcal{H}_0$.
As an application, we derive the following corollary.

\begin{corollary}\label{corollary}
Let $k \ge 5$ and $r \le k-2$.
Let $\mathcal{H}$ be a $\emph{Berge-}M_{k+1}$-free $r$-graph on $n$ vertices.
For $n$ sufficiently large, if 
\[
|E(\mathcal{H})| > h_r(n,k,k-1),
\] 
then  $\mathcal{H}$ is a subgraph of $\mathcal{H}^r(n,k,k)$.
\end{corollary}

The paper is organized as follows. 
We first introduce some preliminaries in Section~2.
In Section~3, we first establish a key lemma, then prove Theorem~\ref{stabilitytheorem-for-bergematching}, 
and finally use this theorem to prove Theorem~\ref{berge-matching}.
In Section~4, we prove Theorem~\ref{stabilitytheorem-for-bergematching-sufficiently} and Corollary~\ref{corollary}.

\section{Preliminaries}

Let $\mathcal{H}$ be an $r$-graph and $V \subseteq V(\mathcal{H})$ be a set of vertices.
The subgraph of $\mathcal{H}$ induced by $V$ is denoted by $\mathcal{H}[V]$. 
For convenience, we write $[n]=\{1,\cdots,n\}$.

The Tutte-Berge Theorem characterizes the structure of graphs with bounded matching number.

\begin{theorem}[Tutte-Berge Theorem \cite{Lovász2009}]\label{Tutte-Berge Theorem}
A graph $G$ is $M_{k+1}$-free if and only if there exists a set $T \subseteq V(G)$ 
such that each connected component $C_1,\dots,C_m$ of $G \setminus T$ has an odd number of vertices and
\[
|T| + \sum_{i=1}^{m} \left\lfloor \frac{|V(C_i)|}{2} \right\rfloor \le k.
\]
\end{theorem}

The following lemma can be extracted from Gerbner, Methuku, and Palmer \cite{Gerbnerm2019}.

\begin{lemma}[Gerbner, Methuku, and Palmer \cite{Gerbnerm2019}]\label{gerbnerm2019}
If $\mathcal{H}$ is an $r$-uniform $\text{Berge-}F$-free hypergraph, then there exists an $F$-free red-blue graph $G$ 
and a hyperedge set $E \subseteq E(\mathcal{H})$ such that 
$|E| = e(G_{\text{red}})$ and 
$E(\mathcal{H}) \setminus E \subseteq E(\mathcal{G}^r_{\text{blue}})$.
Consequently,
\[
|E(\mathcal{H})| \le g_r(G).
\]
\end{lemma}

Kang, Ni, and Shan \cite{Kang2022} provided a new proof of the above lemma and also characterized the corresponding extremal graphs. We now present three lemmas due to Kang, Ni, and Shan \cite{Kang2022}.

\begin{lemma}[Kang, Ni, and Shan \cite{Kang2022}]\label{kang2022gr}
For positive integers $n,r$, we have
$$\max\{g_r(G):G\,\, \text{is}\,\, \text{a}\,\, \text{red-}\text{blue}\,\,\text{graph}\,\,\text{of}\,\,\text{complete}\,\,\text{graph}\,\,K_n\}=\max\left\{ {n\choose 2},{n \choose r} \right\}.$$
Furthermore, the equality holds if and only if all the edges of $K_n$ are blue when $r\leq n-3$, all the edges of $K_n$ are red when $r\geq n-1$, 
and all edges are either all blue or all red when $r=n-2$.
\end{lemma}

\begin{lemma}[Kang, Ni, and Shan \cite{Kang2022}]\label{kang2022}
For positive integers $p, r,\ell$ with $r\leq p-1$ and $0\leq \ell \leq p$,
$$\ell +{p-\ell \choose r-1}\leq {p \choose r-1},$$
with equality if and only if $r=2$, or $\ell =0$ and $r\geq 3$.
\end{lemma}

\begin{lemma}[Kang, Ni, and Shan \cite{Kang2022}]\label{kang2022max}
For positive integers $p, r,\ell$ with $\ell\leq p$,
$$\ell +{p-\ell \choose r-1}\leq \max\left\{ {p\choose r-1},p \right\}.$$
\end{lemma}

\section{Proofs of Theorem \ref{stabilitytheorem-for-bergematching}}

We begin with a key lemma that will be useful in the proof of 
Theorem~\ref{stabilitytheorem-for-bergematching}.

\begin{lemma}\label{mainlemma}
Given a graph $H(n,k,t,\vec{c})$, we have
\begin{equation}\label{mainlemma:eq}
\max \Big\{ g_r(G) : G \text{ is a red-blue graph of } H(n,k,t,\vec{c}) \Big\} 
\le \max \big\{ h_r(n,k,t), f_r(n,k,t) \big\}.
\end{equation}
Moreover, if the equality holds, then $H(n,k,t,\vec{c}) = H(n,k,t)$. 
If, in addition, $3 \le r \le t-1$, then equality implies that 
$H(n,k,t,\vec{c}) = H(n,k,t)$ and all edges of $H(n,k,t)$ are blue.
\end{lemma}

\begin{proof}
Recall that $H(n,k,t,\vec{c})$ is the $n$-vertex graph $K_t + K(\vec{c})$, 
where $K(\vec{c})$ is the graph consisting of cliques of order $c_1, \dots, c_m$.
Let $C_1, \dots, C_m$ be the connected components of $K(\vec{c})$ in $G$, 
where $|C_i| = c_i$. 

Let $G$ be a red-blue graph of $H(n,k,t,\vec{c})$ that maximizes $g_r(G)$, i.e.,
\[
g_r(G) = \max \{ g_r(G') : G' \text{ is a red-blue graph of } H(n,k,t,\vec{c}) \}.
\]
We will prove that
\begin{equation}\label{mainlemma:pf:eq1}
g_r(G) \le \max \{ h_r(n,k,t), f_r(n,k,t) \}.
\end{equation}

We consider two cases according to the value of $r$.

\medskip

{\bf Case 1.} $c_1+t\geq r+2$.

\medskip

Let $G^{\prime}$ be the graph obtained from $G$ by coloring all edges of 
$G[V(C_1) \cup V(K_t)]$ blue, and keeping the colors of all remaining edges 
the same as in $G$. 
It follows from Lemma~\ref{kang2022gr} that
\begin{align*}
g_r(G^{\prime}) - g_r(G) 
&\ge g_r\big(G^{\prime}[V(C_1) \cup V(K_t)]\big) - g_r\big(G[V(C_1) \cup V(K_t)]\big) \\
&\ge {c_1+t \choose r} - \max \Big\{ {c_1+t \choose 2}, {c_1+t \choose r} \Big\} \\
&= 0.
\end{align*}
The last equality holds because ${c_1+t \choose r} \ge {c_1+t \choose 2}$ for $c_1+t \ge r+2$.

If $c_i>1$ for some $i\geq 2$, let $a,b \in C_i$ be two vertices. 
Let $B_{a,b}$ be any set of $c_i - 2$ vertices in $C_1$. 
We construct $G_1$ from $G'$ by performing the following steps (see Figure~\ref{fig:2}):
\begin{itemize}
    \item[{\bf Step 1.}] Connect both $a$ and $b$ to $B_{a,b}$, using the same colors as the original edges between $\{a,b\}$ and $C_i \setminus \{a,b\}$;
    \item[{\bf Step 2.}] Connect both $a$ and $b$ to $V(C_1) \setminus B_{a,b}$ and color these new edges red;
    \item[{\bf Step 3.}] Delete all edges between $\{a,b\}$ and $C_i \setminus \{a,b\}$ (i.e., the edges in $C_i$ incident to $a$ or $b$).
\end{itemize}
All other edges retain the same colors as in $G'$.

\begin{figure}[htbp]
    \centering 
    \includegraphics[width=0.7\textwidth]{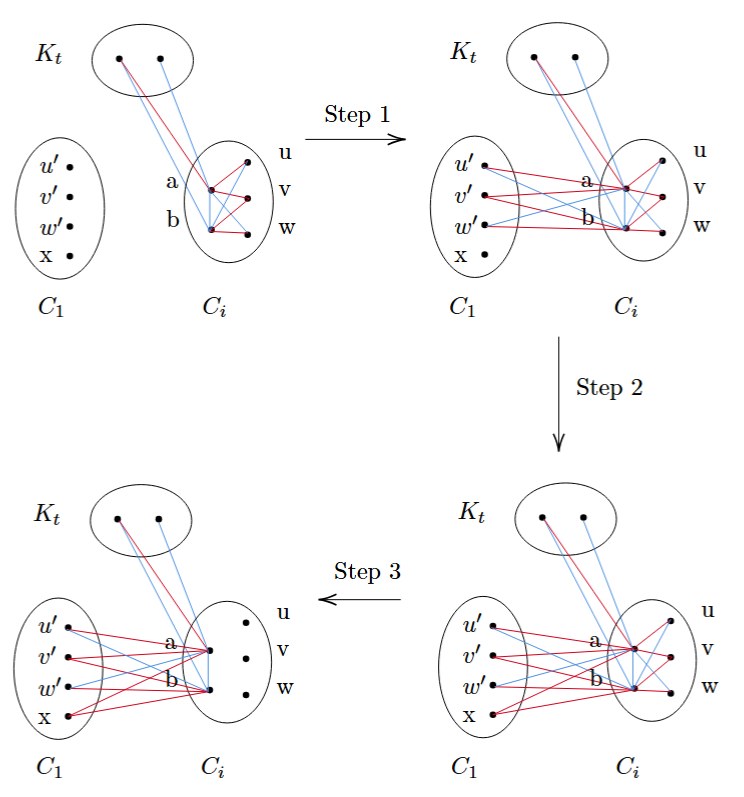} 
    \caption{$G^{\prime} \to G_{1}$} 
    \label{fig:2} 
\end{figure}

Observe that during this process, the decrease in $g_r(\cdot)$ caused by edge deletions (Step~3) is compensated by the increase from the added edges (Steps~1). 
Moreover, since $|V(C_1) \setminus B_{a,b}| = c_1 - (c_i - 2) > 0$, Step~2 strictly increases $g_r(\cdot)$. 
Hence, we have
\[
g_r(G') < g_r(G_1).
\]

Overall, this procedure effectively transfers two vertices from $C_i$ (if $c_i>1$ for some $i\geq 2$) to $C_1$, 
while ensuring that $g_r(\cdot)$ strictly increases throughout the transformation from $G$ to $G'$ to $G_1$.

By iterating the above process repeatedly, we obtain a sequence
\[
G \to G' \to G_1 \to G_1' \to G_2 \to \cdots \to G_\ell
\]
such that
\[
g_r(G) \le g_r(G') < g_r(G_1) \le g_r(G_1') < g_r(G_2) \le \cdots \le g_r(G_{\ell-1}') < g_r(G_\ell),
\]
where $\ell = \sum_{i=2}^{m} \lfloor c_i/2 \rfloor$ and $G_\ell$ is a red-blue graph of $H(n,k,t)$.
Thus,
\begin{equation}\label{mainlemma:pf:case1:eq1}
g_r(G) \le \max \big\{ g_r(F) : F \text{ is a red-blue graph of } H(n,k,t) \big\}.
\end{equation}
Moreover, if equality in (\ref{mainlemma:pf:case1:eq1}) holds, then $c_i=1$ for any $i\geq 2$,
which implies that $G$ is a red-blue graph of $H(n,k,t)$, 
So, $H(n,k,t,\vec{c}) = H(n,k,t)$.

\medskip

Next, we prove the following Claim.

\begin{claim}\label{mainlemma:pf:claim}
\begin{equation}\label{mainlemma:pf:case1:eq2}
\max \big\{ g_r(F) : F \text{ is a red-blue graph of } H(n,k,t) \big\} 
\le \max \big\{ h_r(n,k,t), f_r(n,k,t) \big\}.
\end{equation}
\end{claim}

\begin{proof}
Let $F$ be a red-blue graph of
\[
H(n,k,t) = K_t + \big(K_{2k+1-2t} \cup (n-2k-1+t)K_1 \big).
\]
Let $U = V(K_t)$, $V = V(K_{2k+1-2t})$, and $W = V(H(n,k,t)) \setminus (U \cup V)$. 
For each $w_i \in W$, denote by $\ell_i$ the number of red edges incident to $w_i$ in $F$. 
Then the number of blue $r$-cliques containing $w_i$ in $F$ is at most ${t-\ell_i \choose r-1}$. 
Since $F[U \cup V]$ is a clique of size $2k+1-t$ and $r \le 2k-1-t$, it follows from Lemma~\ref{kang2022gr} that
\[
g_r(F[U \cup V]) \le \max \Big\{ {2k+1-t \choose 2}, {2k+1-t \choose r} \Big\} \le {2k+1-t \choose r}.
\]
Therefore, by the definitions of $U$, $V$, and $W$, we obtain
\begin{eqnarray*}
g_r(F)&\leq& g_r(F[U\cup V])+\sum\limits_{w_i\in W}\left(\ell_i+{t-\ell_i\choose r-1}\right)\\
&\leq& {2k+1-t \choose r}+\sum\limits_{w_i\in W}\max\left\{{t\choose r-1},t \right\}\\
&\leq& {2k+1-t \choose r}+(n-2k-1+t)\cdot \max\left\{{t\choose r-1},t \right\}\\
&=&\max\left\{h_r(n,k,t),f_r(n,k,t) \right\}.
\end{eqnarray*}
The second inequality follows from Lemma \ref{kang2022max}.
By the arbitrariness of $F$, we deduce that (\ref{mainlemma:pf:case1:eq2}) holds. 
\end{proof}

It then follows from (\ref{mainlemma:pf:case1:eq1}) and (\ref{mainlemma:pf:case1:eq2}) that (\ref{mainlemma:pf:eq1}) holds.  
Furthermore, if equality in (\ref{mainlemma:eq}) holds, then equality in (\ref{mainlemma:pf:case1:eq1}) also holds. 
Hence, $H(n,k,t,\vec{c}) = H(n,k,t)$. 
If $3 \le r \le t-1$, then $r \le c_1 + t - 2$. 
Thus, as the same as the proof of Claim \ref{mainlemma:pf:claim}, we further have
\begin{eqnarray*}
g_r(F)&\leq& g_r(F[U\cup V])+\sum\limits_{w_i\in W}\left(\ell_i+{t-\ell_i\choose r-1}\right)\\
&\leq& {2k+1-t \choose r}+\sum\limits_{w_i\in W}{t\choose r-1}\\
&\leq& {2k+1-t \choose r}+(n-2k-1+t)\cdot{t\choose r-1}\\
&=&h_r(n,k,t).
\end{eqnarray*}
The second inequality is due to Lemma \ref{kang2022} under the condition $3 \le r \le t-1$.
We deduce that if equality in (\ref{mainlemma:eq}) holds, then
\begin{eqnarray}\label{mainlemma:pf:case1:eq3}
\sum_{w_i \in W} \Big( \ell_i + {t-\ell_i \choose r-1} \Big) = (n-2k-1+t) \cdot {t \choose r-1} \text{ and }
g_r(F[U \cup V]) = {2k+1-t \choose r}.    
\end{eqnarray}
Thus, $\ell_i + {t-\ell_i \choose r-1} = {t \choose r-1}$ for every $w_i$. 
Note that $3 \le r \le t-1$.
Applying Lemma~\ref{kang2022}, we obtain that $\ell_i = 0$ for all $w_i$ and all edges of $F[U]$ are blue.
Moreover, all edges of $F[U \cup V]$ are blue by Lemma~\ref{kang2022gr} together with (\ref{mainlemma:pf:case1:eq3}). 
In summary, all edges of $H(n,k,t)$ are colored blue.

\medskip

{\bf Case 2.} $c_1+t<r+2$.

\medskip

Since $r \le k-1$ and $c_1 + t < r+2$, we conclude that $t < k$ 
(otherwise, if $t=k$, then $c_1=1$ and $r>k-1$, a contradiction). 
Hence, we deduce that $c_1 \ge 3$ and thus $r > c_1 + t - 2 \ge t + 1$.

Note that the number of $r$-cliques in $G[V(C_i)\cup V(K_t)]$ is at most ${c_i+t \choose r}$. 
Since $3 \le c_1+t < r+2$, it follows that
\[
{c_i+t \choose r} \le 
\begin{cases} 
c_i t, & \text{if } t \neq 0,\\
{c_i+t \choose 2}, & \text{if } t=0.
\end{cases}
\]
Hence, $g_r(G) \le N(K_2, H(n,k,t,\vec{c}))$, 
i.e., all edges of $G$ are red when $r \ge 3$, 
and each edge is either red or blue when $r=2$.

From the definitions of $H(n,k,t,\vec{c})$ and $H(n,k,t)$, it can be readily verified that
\[
N(K_2, H(n,k,t,\vec{c})) \le N(K_2, H(n,k,t))= h_2(n,k,t).
\]
Since $r \le k-1 \le 2k-1-t$, we have ${2k+1-t \choose 2} \le {2k+1-t \choose r}$. 
Then $h_2(n,k,t) \le f_r(n,k,t)$. 
Thus, $g_r(G) \le f_r(n,k,t)$ and (\ref{mainlemma:eq}) holds.

An easy observation shows that if (\ref{mainlemma:eq}) becomes an equality, 
then $N(K_2, H(n,k,t,\vec{c})) = N(K_2, H(n,k,t))$, 
which means that $H(n,k,t,\vec{c}) = H(n,k,t)$.

In both cases, the proof is complete.
\end{proof}

{\bf Proof of Theorem \ref{stabilitytheorem-for-bergematching}:}
If $r=2$, let $G = \mathcal{H}$ and we color all edges of $G$ blue. 
Then $g_r(G) = \vert E(\mathcal{H})\vert$. 
If $r \ge 3$, by Lemma~\ref{gerbnerm2019}, we can construct an $M_{k+1}$-free red-blue graph $G$ 
and a hyperedge set $E \subseteq E(\mathcal{H})$ such that
\[
\vert E\vert = e(G_{\text{red}}), \quad 
E(\mathcal{H}) \setminus E \subseteq E(\mathcal{G}^r_{\text{blue}}), \quad 
\text{and} \quad \vert E(\mathcal{H})\vert \le g_r(G).
\]
Thus,
\begin{equation}\label{stabilitytheorem-for-bergematching:pf:eq1}
g_r(G) \ge \vert E(\mathcal{H})\vert > \max\{h_r(n,k,p), h_r(n,k,k-q), f_r(n,k,p), f_r(n,k,k-q)\}.
\end{equation}

By Theorem~\ref{Tutte-Berge Theorem}, there exists a set $T \subseteq V(G)$ 
such that each connected component $C_1, \dots, C_m$ of $G \setminus T$ is odd and
\[
\vert T\vert + \sum_{i=1}^{m} \left\lfloor \frac{\vert V(C_i)\vert}{2} \right\rfloor \le k.
\]
Let $t = \vert T\vert$ and $c_i = \vert V(C_i)\vert$. 
With loss of generality, we label the components so that $c_1 \ge \cdots \ge c_m$. 
Clearly, $0 \le t \le k$. 
Moreover, 
\[
\alpha(\vec{c}) = \sum_{i=1}^{m} c_i = n - t, \quad
\beta(\vec{c}) = \sum_{i=1}^{m} \left\lfloor \frac{c_i}{2} \right\rfloor \le k - t,
\]
and each $c_i$ is odd for $1 \le i \le m$.

If $\beta(\vec{c}) = k-t$, then $G \setminus T$ is a subgraph of $K(\vec{c})$ and $G$ is a subgraph of $H(n,k,t,\vec{c})$.
We will therefore consider the case $\beta(\vec{c}) < k-t$.
Note that
\[
m = \alpha(\vec{c}) - 2\beta(\vec{c}) = n - t - 2\beta(\vec{c}) 
\ge 2k + 2 - t - 2\beta(\vec{c}) \ge 2\bigl[(k-t) - \beta(\vec{c})\bigr] + 2.
\]
Let 
\[
G' = \bigl(C_1 + C_2 + \cdots + C_{2\cdot[(k-t)-\beta(\vec{c})]} + C_{2\cdot[(k-t)-\beta(\vec{c})]+1}\bigr) 
\cup C_{2\cdot[(k-t)-\beta(\vec{c})]+2} \cup \cdots \cup C_m
\]
and
\[
\vec{c}^{\,\prime} = \bigl(c_1 + c_2 + \cdots + c_{2\cdot[(k-t)-\beta(\vec{c})]} + c_{2\cdot[(k-t)-\beta(\vec{c})]+1}, \,
c_{2\cdot[(k-t)-\beta(\vec{c})]+2}, \dots, c_m\bigr).
\]
It is easy to check that $G \setminus T \subseteq G'$, $\alpha(\vec{c}^{\,\prime}) = n-t$, $\beta(\vec{c}^{\,\prime}) = k-t$, 
and each component in $\vec{c}^{\,\prime}$ is an odd number. 
Thus, $G'$ is a subgraph of $K(\vec{c}^{\,\prime})$, and $G$ is a subgraph of $H(n,k,t,\vec{c}^{\,\prime})$. 
Therefore, without loss of generality, we may assume that $G$ is a subgraph of $H(n,k,t,\vec{c})$ with $\beta(\vec{c}) = k-t$.

The rest of this proof only needs to show that $0 \leq t \leq p-1$ or $k-q+1 \leq t \leq k$.
If $p \leq t \leq k-q$, then using the convexity of $h_r(n,k,t)$ and $f_r(n,k,t)$, we conclude from Lemma \ref{mainlemma} that
\begin{eqnarray*}
g_r(G) &\leq& \max\{h_r(n,k,t), f_r(n,k,t)\} \\
&\leq& \max\{h_r(n,k,p), h_r(n,k,k-q), f_r(n,k,p), f_r(n,k,k-q)\},
\end{eqnarray*}
which contradicts (\ref{stabilitytheorem-for-bergematching:pf:eq1}). 
Hence, $0 \leq t \leq p-1$ or $k-q+1 \leq t \leq k$.
Namely, we can construct a graph $G$ that satisfies the requirements.
\hfill \rule{3mm}{3mm}\par\medskip

{\bf Proof of Theorem \ref{berge-matching} using Theorem \ref{stabilitytheorem-for-bergematching}:}
Since both $\mathcal{H}^r(n,k,0)$ and $\mathcal{H}^r(n,k,k)$ are $\mbox{Berge-}M_{k+1}$-free,
we have 
$$
\mathrm{ex}_r(n,\mbox{Berge-}M_{k+1})\geq 
\max\left\{ h_r(n,k,0),h_r\left(n,k,k\right) \right\}.
$$
Let $\mathcal{H}$ be an extremal $r$-graph on $n$ vertices with respect to $\mbox{Berge-}M_{k+1}$.
That is,
\begin{eqnarray}\label{pf:Theorem1.4:eq1}
E(\mathcal{H})\geq \max\left\{ h_r(n,k,0),h_r\left(n,k,k\right) \right\}.
\end{eqnarray}

By the definitions and the convexity of $h_r(n,k,t)$ and $f_r(n,k,t)$, we have
\begin{eqnarray*}
&&\max\left\{ h_r(n,k,0),h_r\left(n,k,k\right) \right\}\\
&=&\max\left\{ h_r(n,k,0),h_r\left(n,k,k\right),f_r(n,k,0),f_r\left(n,k,k\right) \right\}\\
&>& \max\left\{ h_r(n,k,1),h_r\left(n,k,k-1\right),f_r(n,k,1),f_r\left(n,k,k-1\right) \right\}.
\end{eqnarray*}
It follows from Theorem \ref{stabilitytheorem-for-bergematching} that 
there exists a hyperedge set $E\subseteq E(\mathcal{H})$ and an $M_{k+1}$-free red-blue graph $G$, 
which is a subgraph of $H(n,k,0,\vec{c})$ or $H(n,k,k,\vec{c})=H(n,k,k)$,
such that $\vert E\vert = e(G_{\text{red}})$, $E(\mathcal{H}) \setminus E\subseteq E(\mathcal{G}^r_{\text{blue}})$, and 
$E(\mathcal{H})\leq g_r(G)$.
In particular, $G = \mathcal{H}$ when $r=2$.

Furthermore, if $G$ is a subgraph of $H(n,k,0,\vec{c})$, then by (\ref{pf:Theorem1.4:eq1})
we conclude that $G$ is a subgraph of $H(n,k,0)$.
Thus, $G$ is a subgraph of either $H(n,k,0)$ or $H(n,k,k)$.
By Lemma \ref{mainlemma}, we have 
$$ 
g_r(G)\leq \max\left\{ h_r(n,k,0),h_r\left(n,k,k\right) \right\}.
$$
Combining this with (\ref{pf:Theorem1.4:eq1}), we get
$$
E(\mathcal{H})= g_r(G)= \max\left\{ h_r(n,k,0),h_r\left(n,k,k\right) \right\}.
$$

\medskip

If $r=2$, then $\mathcal{H}=G\in \{H(n,k,0),H(n,k,k)\}$.
If $3\leq r\leq k-1$, then all edges of $G$ are blue by Lemma \ref{mainlemma}.
Therefore, $\mathcal{H}\in \{\mathcal{H}^r(n,k,0),\mathcal{H}^r(n,k,k)\}$.
\hfill \rule{3mm}{3mm}\par\medskip

\section{Proofs of Theorem \ref{stabilitytheorem-for-bergematching-sufficiently} and Corollary \ref{corollary}}
\textbf{Proof of Theorem \ref{stabilitytheorem-for-bergematching-sufficiently}.}
Since $n$ is sufficiently large, we have
\[
h_r(n,k,k-q)=\max\{h_r(n,k,0),\, h_r(n,k,k-q),\, f_r(n,k,0),\, f_r(n,k,k-q)\}.
\]
By Theorem \ref{stabilitytheorem-for-bergematching},
there exist a hyperedge set $E\subseteq E(\mathcal{H})$ and an $M_{k+1}$-free red-blue graph $G$, that is a subgraph of $H(n,k,t,\vec{c})$ for some $(t,\vec{c})$ satisfying $k-q+1\le t\le k$
such that $|E|=e(G_{\text{red}})$ and $E(\mathcal{H})\setminus E \subseteq E(\mathcal{G}^r_{\text{blue}})$.
Moreover, $G=\mathcal{H}$ and $E=\emptyset$ when $r=2$.
Let $N_{\mathcal{H}}(K_r,G_{\text{blue}})$ denote the number of $r$-cliques in $G_{\text{blue}}$ that can be formed from the edges in $E(\mathcal{H})$.
Thus,
\[
|E(\mathcal{H})|
  = e(G_{\text{red}})+N_{\mathcal{H}}(K_r,G_{\text{blue}})
  \le e(G_{\text{red}}) + N(K_r, G_{\text{blue}})=g_r(G).
\]

Given a graph $H(n,k,t,\vec{c})$, let 
$m'=\max\{\, i : c_i \ge 2 \,\}$, 
$U = V(K_t)$, 
$V = V\!\left(\bigcup_{i=1}^{m'} K_{c_i}\right)$, 
and 
$W = V\!\left(\bigcup_{i=m'+1}^{m} K_{c_i}\right)$.
Let $S$ be the subset of $U$ obtained by iteratively removing vertices until 
the remaining set satisfies the property that each vertex $s \in S$ is contained in at least
\[
\left(r(k+1)+t+1\right)\binom{t}{r-1}
\]
distinct blue $r$-cliques in $G[S \cup W]$ that can be formed from the edges in $E(\mathcal{H})$.
Obviously, $\vert S\vert \leq \vert U\vert\leq t\leq k$.
Without loss of generality, assume that $|S|=\gamma$, 
$S = \{u_1,\ldots,u_{\gamma}\}$, 
and 
$U \setminus S = \{u_{\gamma+1},\ldots,u_{t}\}$.

\medskip

{\bf Claim. }  $\gamma \geq k-q$.

\medskip

\begin{proof}
Otherwise, we suppose $\gamma \le k - q - 1$.
Define $$g^{\prime}_r(G) = e(G_{\text{red}}) + N_{\mathcal{H}}(K_r, G_{\text{blue}}).$$ 
Hence, $\lvert E(\mathcal{H}) \rvert = g^{\prime}_r(G) \le g_r(G)$.
We consider the following two cases.

\medskip

\textbf{Case 1. } $3 \le r \le k - q - 1$.

\medskip

We estimate $g^{\prime}_r(G)$ by decomposing it into three parts.

\textbf{Part 1:} Estimate of $g^{\prime}_r(G[U \cup V])$.
Since $\lvert U \cup V \rvert = O(1)$, it follows that $g^{\prime}_r(G[U \cup V]) \le g_r(G[U \cup V]) = O(1)$.

\medskip

\textbf{Part 2:} Contribution from edges between $W$ and $U \setminus S$.
Recall that $W$ is an independent set.
For each vertex $u_i \in U \setminus S$, the number of blue $r$-cliques containing $u_i$ that may arise from edges in $E(\mathcal H)$ is bounded above by
\[
\bigl(r(k+1) + t + 1\bigr)\binom{t}{r-1} - 1.
\]
Let $\eta_i$ denote the number of red edges joining $u_i$ to vertices in $W$ for $\gamma + 1 \le i \le t$.
Therefore, the total contribution of vertices in $U \setminus S$ to $g'_r(G)$ is at most
\[
\sum_{u_i \in U \setminus S}
\left[
\bigl(r(k+1) + t + 1\bigr)\binom{t}{r-1}
- 1 + \eta_i
\right].
\]

\medskip

\textbf{Part 3:} Contribution from vertices in $W$ that do not involve any vertex from $U \setminus S$.
For each $w_i \in W$, let $\ell_i$ be the number of red edges between $w_i$ and vertices in $S$.
Then the number of blue $r$-cliques containing $w_i$ but avoiding all vertices from $U \setminus S$ is at most
\[
\binom{\gamma - \ell_i}{r - 1}.
\]
Hence, the total contribution to $g'_r(G)$ from those portions of $W$ not incident to $U \setminus S$ is bounded by
\[
\sum_{w_i \in W}
\left(
\ell_i + \binom{\gamma - \ell_i}{r - 1}
\right).
\]

Therefore, easy calculations show that 
\begin{eqnarray*}
g^{\prime}_r(G)&\leq& \sum_{u_i\in U\setminus S}\left[\left(r(k+1)+t+1\right)\cdot{t \choose r-1}-1+\eta_i\right]\\
&&+\sum_{w_i\in W}\left({\ell_i+{\gamma-\ell_i\choose r-1}}\right)+O(1)\nonumber\\
&\leq&\sum_{u_i\in U\setminus S}\eta_i+\sum_{w_i\in W}\left({\ell_i+{\gamma-\ell_i\choose r-1}}\right)+O(1)\nonumber\\
&\leq&(t-\gamma)n+\sum_{w_i\in W}\left({\ell_i+{\gamma-\ell_i\choose r-1}}\right)+O(1)\\
&\leq& (t-\gamma)n+\sum_{w_i\in W}\max\left\{{\gamma\choose r-1},\gamma\right\}+O(1)\nonumber\\
&\leq& (t-\gamma)n+n\cdot \max\left\{{\gamma\choose r-1},\gamma\right\}+O(1)\nonumber\\
&=&n\cdot\max\left\{{\gamma\choose r-1}+t-\gamma,t\right\}+O(1).
\end{eqnarray*}
The fourth inequality is based on Lemma \ref{kang2022max}.
It is easy to check that $\binom{\gamma}{r-1} - \gamma$ is a convex function of $\gamma$, 
so we have
\begin{eqnarray*}\label{pf:Theorem1.5:claim:case1:eq1}
\binom{\gamma}{r-1} - \gamma \leq \binom{k-q-1}{r-1} - (k-q-1).
\end{eqnarray*}
Since $t\leq k$, $3 \leq r \leq k-q-1$, and $k \geq 2q+3$, we get
\begin{eqnarray}\label{pf:Theorem1.5:claim:case1:eq2}
\binom{\gamma}{r-1} - \gamma+ t  \leq \binom{k-q-1}{r-1} - (k-q-1) + t \leq \binom{k-q-1}{r-1} + q + 1 < \binom{k-q}{r-1}
\end{eqnarray}
and
\begin{eqnarray}\label{pf:Theorem1.5:claim:case1:eq4}
\binom{k-q}{r-1} \geq \frac{(k-q)(k-q-1)}{2} \geq \left( k - \frac{k-3}{2} \right) \cdot \frac{q+2}{2} > k \geq t.
\end{eqnarray}
Combining (\ref{pf:Theorem1.5:claim:case1:eq2}) and (\ref{pf:Theorem1.5:claim:case1:eq4}), 
we obtain
\[
\max\left\{ \binom{\gamma}{r-1} + t - \gamma, t \right\} \leq \binom{k-q}{r-1} - 1.
\]
Hence,
\[
g^{\prime}_r(G) \leq n \left( \binom{k-q}{r-1} - 1 \right) + O(1),
\]
which is a contradiction to $g^{\prime}_r(G)=\vert E(\mathcal{H})\vert>h_r(n,k,k-q)$ for sufficiently large $n$.

\medskip

{\bf Case 2. }  $r=2$.

\medskip

Then $G = \mathcal{H}$ and $E = \emptyset$.
As the proof proceeds similarly to Case 1, we have
\begin{eqnarray*}
g^{\prime}_r(G) \leq \sum_{u_i \in U \setminus S} \eta_i 
+ \sum_{w_i \in W} \left( \ell_i + \binom{\gamma-\ell_i}{r-1} \right) + O(1).
\end{eqnarray*}
Since $E = \emptyset$, it follows that $\eta_i = 0$ and $\ell_i = 0$. 
Therefore,
\begin{eqnarray*}
g^{\prime}_r(G) \leq \sum_{w_i \in W} \binom{\gamma}{r-1} + O(1)
\leq n \binom{\gamma}{r-1} + O(1)
\leq n \binom{k-q-1}{r-1} + O(1).
\end{eqnarray*}
This yields a contradiction to the assumption that $|E(\mathcal{H})| = g^{\prime}_r(G)$ for sufficiently large $n$.

In both cases, the Claim holds.
\end{proof}

Let $V(\mathcal{H}_0) = V(\mathcal{H}) \setminus S$ and 
$E(\mathcal{H}_0) = \{ e \setminus S : e \in E(\mathcal{H}), \, |e \cap (V(\mathcal{H}) \setminus S)| \ge 2 \}$.
It suffices to show that
\[
\mathcal{H}_0 \text{ is a Berge-} M_{k+1-\gamma} \text{-free hypergraph.}
\]

Suppose, for the sake of contradiction, that there exists a subgraph 
$\mathcal{H}_0' \subseteq \mathcal{H}_0$ which is a Berge-$M_{k+1-\gamma}$.
By the definition of a Berge-$F$, there exists a bijection
\[
\phi: E(M_{k+1-\gamma}) \to E(\mathcal{H}_0')
\]
such that $e' \subseteq \phi(e')$ for each $e' \in E(M_{k+1-\gamma})$.
For each $e_0 \in E(\mathcal{H}_0)$, by the definition of $\mathcal{H}_0$, there exists $e \in E(\mathcal{H})$ such that $e_0 \subseteq e$.
Hence we can define a bijection
\[
\psi: E(\mathcal{H}_0') \to E(\mathcal{H}')
\]
such that $e_0 \subseteq e = \psi(e_0)$ for each $e_0 \in E(\mathcal{H}_0')$, where $\mathcal{H}'$ is a subgraph of $\mathcal{H}$.
Composing $\phi$ and $\psi$, we obtain a bijection
\[
\psi \circ \phi: E(M_{k+1-\gamma}) \to \psi(E(\mathcal{H}_0'))
\]
such that $e' \subseteq \psi(\phi(e'))$ for each $e' \in E(M_{k+1-\gamma})$.

Note that $\vert V(\mathcal{H}_0') \vert \le r(k+1-\gamma) \le r(k+1)$.
Hence, $\vert V(\mathcal{H}_0') \cap W \vert \le r(k+1)$.
Since $W$ is an independent set, the number of blue $r$-cliques containing $u_1$ and vertices from $V(\mathcal{H}_0') \cap W$, together with the number of blue $r$-cliques in $G[U]$ containing $u_1$, is at most
\[
\vert V(\mathcal{H}_0') \cap W \vert {t \choose r-2} + {t-1 \choose r-1} \le (r(k+1)+1) \cdot {t \choose r-1}.
\]
Denote by $K_r(x,y)$ an $r$-clique containing vertices $x$ and $y$.
By the Claim and the definition of $S$, we can choose a blue $r$-clique $K_r(u_1,w^1)$ corresponding to a hyperedge $\mathcal{G}^r[K_r(u_1,w^1)]$ in $\mathcal{H}$, 
where $w^1 \in W \setminus V(\mathcal{H}^{\prime}_0)$ and $\mathcal{G}^r[K_r(u_1,w^1)]$ is the $r$-clique hypergraph of $G[V(K_r(u_1,w^1))]$.

For $2 \le i \le \gamma$, similarly, the number of blue $r$-cliques containing $u_i$ and vertices from 
$(V(\mathcal{H}_0') \cap W) \cup \bigcup_{j<i} w^j$,
together with the number of blue $r$-cliques in $G[U]$ containing $u_i$, is at most
\[
(\vert V(\mathcal{H}_0') \cap W \vert + (i-1)) {t \choose r-2} + {t-1 \choose r-1} \le (r(k+1) + t) {t \choose r-1}.
\]
We then pick a blue $r$-clique $K_r(u_i,w^i)$ corresponding to a hyperedge $\mathcal{G}^r[K_r(u_i,w^i)]$ in $\mathcal{H}$,
where $w^i \in W \setminus (V(\mathcal{H}_0') \cup \bigcup_{j<i} w^j)$.

Now we define a bijection 
$$\varXi : E(M_{k+1-\gamma})\cup \big(\bigcup_{1\leq i\leq \gamma} u_iw^i\big) \to \psi(E(\mathcal{H}^{\prime}_0))\cup \big(\bigcup_{1\leq i\leq \gamma}\mathcal{G}^r[K_r(u_i,w^i)]\big)$$ such that 
\begin{eqnarray*}
  \varXi(f)=\begin{cases}
    \psi(\phi(f)) & \text{if }  f\in E(M_{k+1-\gamma}),\\
    \mathcal{G}^r[K_r(u_i,w^i)]        & \text{if }  f=u_iw^i.
    \end{cases}
\end{eqnarray*}
Note that 
\[
\psi(E(\mathcal{H}_0')) \cup \bigcup_{1\le i \le \gamma} \mathcal{G}^r[K_r(u_i,w^i)]
\subseteq E(\mathcal{H}).
\]
Thus, $\mathcal{H}$ contains $\mbox{Berge-}M_{k+1}$ as a subgraph, leading to a contradiction. 
Therefore, $\mathcal{H}_0$ must be $\mbox{Berge-}M_{k+1-\vert S\vert}$-free.
\hfill \rule{3mm}{3mm}\par\medskip

{\bf Proof of Corollary \ref{corollary}:}
By Theorem \ref{stabilitytheorem-for-bergematching-sufficiently}, 
there exists a vertex set $S \subseteq V(\mathcal{H})$ with size $k-1 \leq \vert S \vert \leq k$  
such that $\mathcal{H}_0$ is a $\mbox{Berge-}M_{k+1-\vert S \vert}$-free hypergraph, 
where 
\[
V(\mathcal{H}_0) = V(\mathcal{H}) \setminus S, \quad 
E(\mathcal{H}_0) = \{ e \setminus S : e \in E(\mathcal{H}), \ \vert e \cap (V(\mathcal{H}) \setminus S) \vert \ge 2 \}.
\]

If $\vert S \vert = k$, then $\vert E(\mathcal{H}_0) \vert = 0$. 
This implies that $\mathcal{H}$ is a subgraph of $\mathcal{H}^r(n,k,k)$.
Hence, we may assume that $\vert S \vert = k-1$. 
Then $\mathcal{H}_0$ is a $\mbox{Berge-}M_2$-free hypergraph. 
In particular, we must have 
\begin{eqnarray}\label{pf:corollary:eq1}
\vert V(e_i) \setminus V(e_j) \vert \le 1 \quad \text{for any } e_i, e_j \in E(\mathcal{H}_0).
\end{eqnarray}
Otherwise, if 
$\lvert V(e_i) \setminus V(e_j) \rvert \ge 2$ 
for some $e_i, e_j \in E(\mathcal{H}_0)$, 
then these two edges would form a copy of $\mbox{Berge-}M_2$ in $\mathcal{H}_0$, 
a contradiction.

Since $\vert E(\mathcal{H}) \vert > h_r(n,k,k-1)$ and the definition of $\mathcal{H}_0$, it can be readily verified that
\begin{eqnarray}\label{pf:corollary:eq2}
\vert E(\mathcal{H}_0) \vert&=&\vert E(\mathcal{H}) \vert
-\left\vert \{ e : e \in E(\mathcal{H}), \ \vert e \cap (V(\mathcal{H}) \setminus S) \vert =1 \} \right\vert\nonumber\\
&&-\left\vert \{ e: e \in E(\mathcal{H}), \ \vert e \cap (V(\mathcal{H}) \setminus S) \vert =0 \} \right\vert\nonumber\\
&>&h_r(n,k,k-1)-{k-1 \choose r}-(n-k+1){k-1 \choose r-1}\nonumber\\
&=& 3 {k-1 \choose r-2} + {k-1 \choose r-3}.
\end{eqnarray}
It is worth noting that ${k-1 \choose r-3}=0$ if $r=2$.
Let $$V_m=\max\{\vert V(e_i)\vert:e_i\in E(\mathcal{H}_0)\} \text{ and } \vert V(e_m)\vert =V_m.$$

The rest of the proof will be divided into three cases.

\medskip

{\bf Case 1. }  $V_m\geq 4$.

\medskip

Since $\mathcal{H}_0$ is $\mathrm{Berge}\text{-}M_2$-free, no two edges of $\mathcal{H}_0$ can satisfy 
$\lvert V(e_i) \setminus V(e_j) \rvert \ge 2$. 
However, if $V_m \ge 4$, then any two distinct edges with maximum vertex size would violate this condition. 
Thus $\lvert E(\mathcal{H}_0) \rvert = 1$, 
which contradicts \eqref{pf:corollary:eq2}.

\medskip

{\bf Case 2.} $V_m = 3$.

\medskip

It follows from the fact that $\mathcal{H}_0$ is a $\mbox{Berge-}M_2$-free hypergraph that 
$V(e_i) \subseteq V(e_m)$ for any $e_i \in E(\mathcal{H}_0)$. 
Consequently,
\[
E(\mathcal{H}_0) = \{ \rho_0 \{v_1, v_2, v_3\}, \ \rho_1 \{v_2, v_3\}, \ \rho_2 \{v_1, v_3\}, \ \rho_3 \{v_1, v_2\} \},
\]
where all $v_i \in V(\mathcal{H}_0)$ are distinct, and each $\rho_i$ is a constant.
By the definition of $\mathcal{H}_0$, we have 
\[
\rho_0 \le {k-1 \choose r-3}, \quad \rho_1, \rho_2, \rho_3 \le {k-1 \choose r-2}.
\]
Therefore,
\[
\vert E(\mathcal{H}_0) \vert \le 3 {k-1 \choose r-2} + {k-1 \choose r-3},
\]
which is a contradiction to (\ref{pf:corollary:eq2}).

\medskip

{\bf Case 3. }  $V_m=2$.

\medskip

Since (\ref{pf:corollary:eq2}) holds, there exist at least four distinct edges $e_i$ 
such that $e_i \in E(\mathcal{H}_0)$ for $i \in [4]$.
Using (\ref{pf:corollary:eq1}) that $\vert V(e_i) \setminus V(e_j) \vert \le 1$ for any $e_i, e_j \in E(\mathcal{H}_0)$ together with the fact that $\mathcal{H}_0$ is a $\mbox{Berge-}M_{2}$-free hypergraph,  
it is then straightforward to verify that $\mathcal{H}_0$ forms a star.
This implies that $\mathcal{H}$ is a subgraph of $\mathcal{H}^r(n,k,k)$. 

Hence, the proof is complete.
\hfill \rule{3mm}{3mm}\par\medskip

\section*{Declaration of competing interest}
The authors declare that they have no known competing financial interests or personal
relationships that could have appeared to influence the work reported in this paper.

\section*{Data availability}
No data was used for the research described in the article.

\section*{Acknowledgements}
The authors would like to thank Long-Tu Yuan for fruitful discussions.


\begin{thebibliography}{9}
\bibitem{Alon2012} N. Alon, P. Frankl, H. Huang, V. R\"odl, A. Ruci\'nski, and B. Sudakov, Large matchings in uniform hypergraphs and the conjecture of Erd\H{o}s and Samuels, {\em J. Combin. Theory Ser. A} {\bf 119} (2012) 1200-1215.

\bibitem{Bollobas1976} B. Bollob{\'a}s, D. Daykin, and P. Erd\H{o}s, Sets of independent edges of a hypergraph, {\em Quart. J. Math.} {\bf 27} (1976) 25-32.

\bibitem{Duan2020} X. Duan, B. Ning, X. Peng, J. Wang and W. Yang, Maximizing the number of cliques in graphs with given matching number, {\em Discrete Appl. Math.} {\bf 287}, (2020), 110-117.

\bibitem{Erdos1965} P. Erd\H{o}s, A problem on independent $r$-tuples, {\em Ann. Univ. Sci. Budapest. E\"otv\"os Sect. Math.} {\bf 8} (1965) 93-95.

\bibitem{Erdos1959} P. Erd\H{o}s and T. Gallai, On maximal paths and circuits of graphs, {\em Acta Math. Acad. Sci. Hungar.} {\bf 10} (1959) 337-356.

\bibitem{Frankl2017} P. Frankl, On maximum number of edges in a hypergraph with given matching number, {\em Discrete Appl. Math.} {\bf 216} (2017) 562-581.

\bibitem{Frankl2022} P. Frankl and A. Kupavskii, The Erd\H{o}s matching conjecture and concentration inequalities, {\em J. Combin. Theory Ser. B} {\bf 157} (2022) 366-400.

\bibitem{Frankl2019} P. Frankl and A. Kupavskii, Two problems on matchings in set families-in the footsteps of Erd\H{o}s and Kleitman, {\em J. Combin. Theory Ser. B} {\bf 138} (2019) 286-313.

\bibitem{Frankl2012} P. Frankl, T. \L uczak, and K. Mieczkowska, On matchings in hypergraphs, {\em Electron J. Combin.} {\bf 19} (2012) R42.

\bibitem{Gerbnerm2019} D. Gerbner, A. Methuku, and M. Vizer, Asymptotics for the Turán number of Berge-$K_{2,t}$, {\em J. Comb. Theory, Ser. B} {\bf 137} (2019) 264-290.

\bibitem{Gerbner2017} D. Gerbner and C. Palmer, Extremal results for Berge-hypergraphs, {\em SIAM J. Discrete Math.} {\bf 31} (2017) 2314-2327.

\bibitem{Kang2022} L. Kang, Z. Ni, and E. Shan, The Tur\'{a}n number of Berge-matching in hypergraphs, {\em Discrete Math.} {\bf 345 (8)} (2022) 112901.

\bibitem{Lovász2009} L. Lovász, M.D. Plummer, Matching Theory, vol. 367, American Mathematical Society, 2009.

\bibitem{Luczak2014} T. \L uczak and K. Mieczkowska, On Erd\H{o}s extremal problem on matchings in hypergraphs, {\em J. Combin. Theory Ser. A} {\bf 124} (2014) 178-194.

\bibitem{Wang2020} J. Wang, The shifting method and generalized Tur\'{a}n number of matching, {\em European J. Combin.} {\bf 85}, 103057 (2020).

\bibitem{Yang2024} J.-B. Yang and L.-T. Yuan, A note on the stability results of the number of cliques in graphs with given matching number, {\em Discrete Appl. Math.} {\bf 356} (2024) 343-349.

\end{thebibliography}
\end{document}